\documentclass[12pt,a4paper,reqno]{amsart}
\usepackage{amscd}
\usepackage{amssymb}
\usepackage{amsthm}
\usepackage{graphicx}
\usepackage{color}
\usepackage[all]{xy}
\usepackage{mathrsfs}
\usepackage{tikz-cd}
\usepackage{mathtools}
\usepackage{hyperref}

 \newtheorem{theorem}{Theorem}[section]

\newtheorem{corollary}[theorem]{Corollary}
\newtheorem{conjecture}[theorem]{Conjecture}

\newtheoremstyle{defstyle}
  {.6em} 
  {.1em} 
  {} 
  {} 
  {\bfseries} 
  {.} 
  {.5em} 
  {} 
\theoremstyle{defstyle} \newtheorem{definition}[theorem]{Definition}

\newtheorem{example}[theorem]{Example}

\theoremstyle{remark}

\numberwithin{equation}{section}
\numberwithin{figure}{section}

\newcommand{\ZZ} {\mathbb{Z}}
\newcommand{\QQ} {\mathbb{Q}}
\newcommand{\RR} {\mathbb{R}}
\newcommand{\CC} {\mathbb{C}}

\def\PP{\mathbb{P}}

\def\Z{\mathbb{Z}}

\begin{document}

\title[Mock modularity of log Gromov--Witten Invariants]{Mock modularity of log Gromov--Witten Invariants: the mirror to $\PP^2$}

\author[H.\,Arg\"uz]{H\"ulya Arg\"uz}
\address{The Mathematical Institute, University of Oxford, Radcliffe Observatory Quarter, Woodstock Road, Oxford, OX2 6GG, U.K}
\email{Hulya.Arguz@maths.ox.ac.uk}

\date{}

\begin{abstract}
We study modularity properties of generating series of logarithmic Gromov–Witten invariants of elliptic fibrations relative to singular fibers. Motivated by predictions from Vafa–Witten theory, we conjecture that such generating series are mock modular forms. We prove this conjecture for a large class of invariants of the rational elliptic surface mirror to $\PP^2$, relative to a cycle of nine rational curves. The proof uses a correspondence between log Gromov–Witten invariants of the mirror and Vafa–Witten invariants of $\PP^2$ established in previous work joint with Bousseau, together with known mock modularity results on the Vafa–Witten side. 
\end{abstract}

\maketitle

\setcounter{tocdepth}{1}
\tableofcontents
\setcounter{section}{0}
\section{Introduction}
\subsection{Background and context}
Modular forms are complex analytic functions distinguished by rich symmetry properties and occupy a central position in modern number theory. Beyond their intrinsic arithmetic significance, they have emerged in unexpected ways across other areas of mathematics and physics. In particular, over the last $30$ years developments in theoretical physics have led to intriguing predictions that certain enumerative invariants, when organized into generating series, naturally give rise to modular forms. For instance, generating series of Gromov–Witten invariants often exhibit striking modularity properties, as anticipated by mirror symmetry. A seminal example of this phenomenon is provided by the work of Okounkov--Pandharipande \cite{OkPa1,OkPa2} and Oberdieck--Pixton \cite{OPelliptic}, who demonstrated that generating series of Gromov–Witten invariants of an elliptic curve are \emph{quasi-modular forms} -- these are complex analytic functions whose failure of modularity is governed by explicit polynomial correction terms.

Gromov–Witten invariants of an elliptic curve $E$, denoted by $GW^E_{g,d}$ enumerate genus $g$ degree $d$ stable maps to $E$, subject to incidence conditions specified by cohomology classes $\alpha = (\alpha_1,\ldots,\alpha_n) \in H^*(E)$ (for simplicity, we suppress the dependence on $\alpha$ in the notation). The associated generating series are defined by summing over all nonnegative degrees:
\[ F^E_{g}(\tau) = \sum_{d \in \ZZ_{\geq 0}}  GW^E_{g,d} q^d \,, \]
where $q= e^{2i\pi \tau }$. The function $F^E_{g}(\tau)$ is quasi-modular; that is, it is holomorphic on the upper half-plane 
\[\mathbb{H} = \{ \tau \in \CC | \mathrm{Im}(\tau) > 0 \} \,, \] 
with some invariance property with respect to the natural action of $\mathrm{SL}_2(\ZZ)$ on $\mathbb{H} $ \cite{OkPa1,OkPa2}. More generally, Oberdieck--Pixton \cite{OP} predict that generating series of Gromov--Witten invariants of \emph{elliptic fibrations} are quasi-modular forms. 

An elliptic fibration is a flat projective morphism $\phi \colon Y \to B$, where $Y$ and $B$ are smooth projective varieties, and the generic fiber of $\phi$ is a smooth elliptic curve; in general, such fibrations may admit singular fibers. The generating series of Gromov--Witten invariants considered in \cite{OP} are formed by summing over curve classes that differ by a multiple of the class of the elliptic fiber.

In \cite{OP}, Oberdieck–Pixton also conjecture the quasi-modularity of generating series of \emph{relative Gromov–Witten invariants} \cite{JL}, which count curves with given tangency orders relative to a smooth divisor generically fibered by smooth elliptic curves -- see Figure \ref{fig:fig3}. Moreover, they prove their main conjectures in several examples, including rational elliptic surfaces.
\begin{figure}[htbp]
\center{\scalebox{.6}{\input{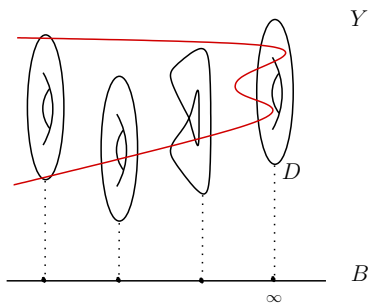}}}
\caption{A rational elliptic surface $Y$ and a curve with tangencies relative to a smooth fiber $D$.}
\label{fig:fig3}
\end{figure}

Logarithmic Gromov–Witten theory, developed by Abramovich--Chen \cite{logGWbyAC,Chen} and Gross--Siebert \cite{logGW}, generalizes the earlier work of Jun Li \cite{JL} on relative Gromov–Witten theory. In this framework, relative incidence conditions may be imposed not only along smooth divisors, but also along divisors with prescribed singularities. In this paper we conjecture \emph{mock modularity} properties of log Gromov--Witten theory of elliptic fibrations $Y \to B$ relative to a singular divisor $D$ generically fibered by singular elliptic curves. 

A mock modular form is a holomorphic function on the upper half-plane whose modularity is recovered only after the addition of a canonical non-holomorphic correction term \cite{zwegers2008mock}. In general, these correction terms can be quite intricate, and their complexity is measured by the \emph{depth} of the mock modular form. In this paper, we study \emph{vector-valued} mock modular forms, that is, vector-valued functions whose components transform under modular transformations as mock modular forms, up to linear combinations among the components; see \S\ref{sec: mock modularity} for details.

While modularity and quasi-modularity have been extensively studied in Gromov–Witten theory, mock modularity has received comparatively little attention in this setting. To our knowledge, mock modularity has thus far appeared only in the study of open Gromov–Witten theory for orbifold quotients of elliptic curves, with no direct connection with logarithmic Gromov–Witten theory; see \cite{bring2, bring1, Lau}.

For elliptic fibrations, viewed as families of elliptic curves, one expects modularity to originate from the modularity of Gromov–Witten invariants of the elliptic curve itself, a phenomenon that can be conceptually explained via mirror symmetry, since the complex moduli space of the mirror elliptic curve is a modular curve, namely the upper half-plane modulo $\mathrm{SL}_2(\ZZ)$; see Dijkgraaf \cite{dijkgraaf1995mirror}. However, in the family setting, especially in the presence of singular fibers and log Gromov–Witten invariants, a conceptual understanding of mock modularity on the log Gromov–Witten side remains elusive. On the Vafa–Witten side, the mathematical picture is similarly incomplete. From the perspective of physics, mock modularity in Vafa–Witten theory is expected to arise as a consequence of the S-duality conjecture, for which a corresponding mathematical explanation is still lacking.

\subsection{Main results}
In this paper we conjecture mock modularity properties of log Gromov--Witten generating series for elliptic fibrations. 

Let $GW^{(Y,D)}_{g,v,\beta}$ denote a log Gromov-Witten count of genus $g$ stable maps in $Y$ of class $\beta \in H_2(Y,\ZZ)$, and with contact order along $D$ prescribed by a vector $v$ with components corresponding to integral points in the tropicalization of $(Y,D)$, given by the dual cone complex of $D$ in $Y$, possibly with insertions of cohomology classes of $Y$ and $D$ (as before, we suppress the insertions from the notation).  

\begin{conjecture}
Let $\phi: Y \to B$ be an elliptic fibration and $D\subset Y$ a normal crossing divisor such that $\phi(D)$ is a divisor in $B$. Assume that the generic fiber of the restricted morphism $\phi|_{D} \colon D \to \phi(D)$ is either a cycle of $\PP^1$'s or a nodal elliptic curve. Then the generating series of logarithmic Gromov–Witten invariants of $(Y,D)$, 
\[ \sum_{d \geq 0} GW^{(Y,D)}_{g,v,\beta_0+dF} q^d \,, \]
where $v$ is a fixed contact order along $D$, $\beta_0 \in H_2(Y,\ZZ)$ is a fixed curve class, and $F$ denotes the fiber class, is a component of a (higher depth) vector-valued mock modular form, up to appropriate normalization in the power of $q$.
\end{conjecture}

Conjecture \ref{conj: mirror P2} is a special case of this conjecture, corresponding to the case where $Y$ is the rational elliptic surface mirror to $\PP^2$. We prove this conjecture in this situation in Theorem \ref{thm:logGW/VW generating} and Corollary \ref{cor:proof of mock log} for a large class of generating series of log Gromov--Witten invariants of $Y$ -- in this setting, $D$ is a fiber of the elliptic fibration on $Y$ consisting of a cycle of nine copies of $\PP^1$ -- see Figure \ref{fig:fig4}. The proof relies on the correspondence between log Gromov–Witten invariants of $(Y,D)$ and Vafa--Witten invariants of $\PP^2$, established in joint work with Bousseau \cite[Theorem 4.1]{ABquiverscurves}.

\begin{figure}[htbp]
\center{\scalebox{.6}{\input{f4.pspdftex}}}
\caption{The rational elliptic surface $Y$, mirror to $\PP^2$, together with a curve having prescribed tangency to the singular fiber $D$.}
\label{fig:fig4}
\end{figure}

For any smooth, projective surface $S$, the Vafa--Witten invariants $VW^S_{r,c_1,c_2}$ are defined using moduli spaces of semistable rank $r$ coherent sheaves on $S$ with first and second Chern classes $c_1$ and $c_2$ \cite{TT, TT1}. In theoretical physics, S-duality of $\mathcal{N} = 4$ super Yang--Mills theory predicts that generating series of Vafa--Witten invariants of the form
\[ \sum_{c_2 \in \Z} VW_{r,c_1,c_2}^S q^{c_2}\]
are, up to a normalization factor of the power of $q$, vector-valued modular forms, when $h^{2,0}(S)>0$, and vector-valued mock modular forms when $h^{2,0}(S)=0$, where $h^{2,0}$ is the $(2,0)$-Hodge number of $S$ \cite{alexandrov2025mock, AP, DPW, Manschot, vafawitten}. Since $h^{2,0}(\PP^2)=0$, the generating series of rank $r$ Vafa–Witten invariants of $\PP^2$ are expected to exhibit mock modular behavior. This expectation was confirmed for $1\leq r \leq 3$ in \cite{Manschot} -- see \S\ref{sec_mock_vw} for a review of the known mock modularity results of Vafa--Witten invariants of $\PP^2$.

The purpose of this paper is to transport this anticipated mock modularity from Vafa–Witten theory to the setting of log Gromov–Witten theory for elliptic fibrations. In Theorem \ref{thm:logGW/VW generating} we prove an equivalence of generating series of Vafa--Witten invariants of $\PP^2$ and log Gromov--Witten invariants of $(Y,D)$, defined in \eqref{eq: vw generating} and \eqref{Eq:log GW generating} respectively. Consequently, in Corollary \ref{cor:proof of mock log} we deduce mock modularity of the log Gromov--Witten generating series. While in Vafa–Witten theory mock modularity originates from four-dimensional gauge theory, on the Gromov–Witten side it manifests as a natural extension of the familiar quasi-modularity phenomena for elliptic fibrations, which arise through mirror symmetry.

\subsection*{Acknowledgement}
I am grateful to Pierrick Bousseau for many helpful discussions, and to Dominic Joyce for posing a question that motivated this work, as well as for his valuable feedback. I also thank the organizers of the Richmond Geometry Meeting for the opportunity to present our research and for organizing the volume to which this paper is submitted. For the purpose of open access, the author has applied a CC BY public copyright licence to any author accepted manuscript arising from this submission.

\section{Mirror symmetry and log Gromov--Witten theory}
We review the definition of log Gromov–Witten invariants, which play a central role in the Gross–Siebert approach to mirror symmetry.
\subsection{The Gross--Siebert program}
The Gross–Siebert program \cite{GSannals,GSintrinsic} provides an algebro-geometric approach to mirror symmetry motivated by the Strominger–Yau–Zaslow conjecture \cite{SYZ}, which, roughly speaking, predicts that mirror pairs of Calabi–Yau varieties admit dual singular torus fibrations. The mirror to a torus fibration on a Calabi–Yau manifold is obtained by compactifying a semi-flat mirror constructed via dualization of the nonsingular torus fibers. Achieving this compactification requires introducing suitable corrections to the complex structure. Kontsevich and Soibelman \cite{KS} in dimension two, and Gross and Siebert \cite{GSannals} in higher dimensions, showed that the determination of these corrections can 
often be reduced to a combinatorial algorithm encoded by a scattering diagram.

In dimension two, Gross, Pandharipande, and Siebert \cite{GPS} further demonstrated that this combinatorial algorithm, producing a \emph{scattering diagram}, admits an interpretation in terms of \emph{log Gromov–Witten invariants} of log Calabi--Yau pairs $(Y,D)$, given by a smooth projective variety $Y$ together with a reduced simple normal crossing divisor $D$ in $Y$ with $K_Y+D=0$. Building on \cite{GPS}, Gross, Hacking, and Keel \cite{GHK} constructed the homogeneous coordinate ring of the mirror to a log Calabi–Yau surface. A generalization of \cite{GPS} to all dimensions is provided in \cite{HDTV} and a general mirror construction in all dimensions in \cite{GSintrinsic}, using scattering diagrams and \emph{punctured log Gromov–Witten theory} of Abramovich--Chen--Gross--Siebert \cite{ACGS}.

\subsection{Log Gromov--Witten invariants}
Given a log Calabi--Yau pair $(Y,D)$, punctured log Gromov--Witten invariants of \cite{ACGS} are counts of curves in $Y$ with prescribed tangency conditions along $D$, which are allowed to be negative. The points at which negative tangency conditions are imposed are referred to as punctured points. Typically, if a map
$f:C\rightarrow Y$ has negative order of tangency with
$D\subseteq Y$ at a punctured point of $C$, then the irreducible
component of $C$ containing this punctured point must map into $D$, and the negative order of tangency is encoded in the data of the log structure. In \cite{HDTV}, we introduce a combinatorial algorithm for computing certain punctured logarithmic Gromov–Witten invariants of a log Calabi--Yau pair $(Y,D)$ in arbitrary dimension. These invariants are defined by counting rational curves with a single punctured marked point mapping to $D$. Such curves—commonly called $\mathbb{A}^1$-curves—play a central role in the Gross–Siebert construction of the canonical scattering diagram, from which mirrors to log Calabi–Yau pairs are produced \cite{HDTV}.

To define counts of $\mathbb{A}^1$-curves in $(Y,D)$, in the particular situation when $Y$ is of complex dimension $2$, we fix $\beta \in H_2(Y,\ZZ) $ and $ v \in \mathrm{Trop}_{\ZZ}(Y,D) $ -- here, 
$\mathrm{Trop}(Y,D)$, called the \emph{tropicalization} of $(Y,D)$ is the dual cone complex of $D$ in $Y$, and $\mathrm{Trop}_{\ZZ}(Y,D) $ is the set of integral points in it. Let $\overline{\mathcal{M}}_{v,\beta}(Y,D)$ denote the moduli space of $\mathbb{A}^1$-curves in $(Y,D)$, whose points correspond to $1$-marked genus $0$ stable log maps to $Y$ with contact order $v$ along $D$ at the marked point. We denote the punctured log Gromov--Witten invariants, corresponding to such counts of $\mathbb{A}^1$-curves in $(Y,D)$ by
\begin{equation} \label{eq_log_gw}
GW_{v,\beta}^{(Y,D)} \coloneqq \int_{[\overline{\mathcal{M}}_{v,\beta}(Y,D)]^{vir}} 1 \in \mathbb{Q} \,. \end{equation}

\section{The rational elliptic surface $(Y,D)$ mirror to $\PP^2$}
\label{sec:rational elliptic}
In this section we review the geometry of the rational elliptic surface $Y$ mirror to $\PP^2$. For more details on the geometry of rational elliptic surfaces that are mirrors to del Pezzo surfaces, see for instance \cite{GGLP}.

Let $\Sigma \subset \RR^2$ be the toric fan whose one-dimensional cones are generated by the primitive integral vectors
\begin{equation*}
\label{eq: rays}
    \mathcal{R} = \{ (1,1),(1,0),(1,-1),(1,-2),(0,-1),(-1,0),(-2,1),(-1,1),(0,1) \} \,. 
\end{equation*}
which are in one-to-one correspondence with the integral lattice points on the boundary of the moment polytope of $\PP^2$ endowed with its anticanonical polarization, as illustrated in Figure \ref{fig:fig1}. Denote by $\overline{Y}$ the corresponding smooth projective toric surface with toric boundary divisor $\overline{D} \subset \overline{Y}$. Note that, since the primitive ray generators of the fan $\Sigma$ of $\overline{Y}$ lie on the boundary of a convex lattice polygon, the anti-canonical divisor $-K_{\overline{Y}} \sim \overline{D}$ is big and nef. Hence, $\overline{Y}$ is a weak del Pezzo surface.


\begin{figure}
\center{\scalebox{.4}{\input{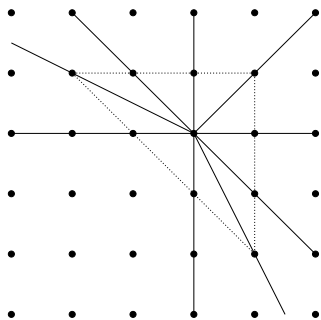}}}
\caption{The fan $\Sigma \subset \RR^2$ for $\overline{Y}$.}
\label{fig:fig1}
\end{figure}

We use the notation $\overline{D}_{(i,j)} $ for the component of $\overline{D}$ corresponding to the ray generated by $(i,j) \in \mathcal{R} $. Pick smooth points of $\overline{D}$, 
\[ x_{(1,1)} \in  \overline{D}_{(1,1)}, \,\  x_{(1,-1)} \in  \overline{D}_{(1,-1)}, \,\  x_{(-2,1)} \in  \overline{D}_{(-2,1)} \,, \]
such that there exists an anti-canonical curve $\overline{C}$ with 
\[ \overline{C} \cap \overline{D} = \{  x_{(1,1)}, x_{(1,-1)},  x_{(-2,1)} \}  \,, \]
as illustrated in Figure \ref{fig:fig2}. 
Let $Y$ be the blow-up of $\overline{Y}$ at these $3$ points and let $D$ denote the strict transform of $\overline{D}$. Then, the pair $(Y,D)$ is a log Calabi--Yau pair, that is, $K_Y+D = 0$. Moreover, $Y$ is a rational elliptic surface, admitting an elliptic fibration given by the map
\begin{equation}
\label{eq:fibration}
    \phi_{|-K_Y|}  \colon Y \longrightarrow | -K_Y | = \PP^1 \,, 
\end{equation}
induced by the linear system of $|-K_Y|$. Note that the divisor $D\subset Y$ is a fiber of $\phi_{|-K_Y|}$, consisting of a cycle of $9$ copies of $\PP^1$, each with self-intersection $-2$. There are exactly $3$ other singular fibers of $\phi_{|-K_Y|}$, each being a nodal elliptic curve.

\begin{figure}[htbp]
\center{\scalebox{.6}{\input{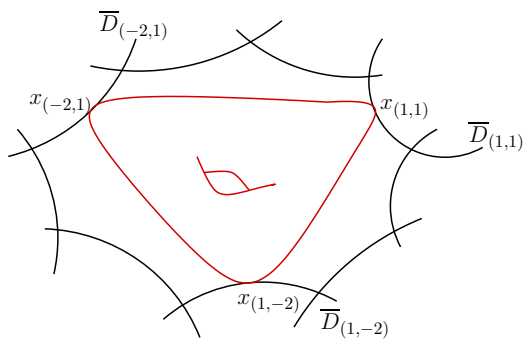}}}
\caption{An anti-canonical curve $\overline{C}$ in $\overline{Y}$.}
\label{fig:fig2}
\end{figure}


It is known that $(Y,D)$ is the \emph{mirror} to $\PP^2$ with a smooth elliptic curve, from several different points of view on mirror symmetry. For instance, the restriction of the fibration $\phi_{|-K_Y|}$ to the torus $(\CC^*)^2\subset Y \setminus D$ is the map 
\begin{align*}
(\CC^*)^2 &  \longrightarrow \CC \\
(x,y)      & \longmapsto x+y+ \frac{1}{xy} \,,
\end{align*}
which is the Hori-Vafa mirror to $\PP^2$ with its toric boundary, and the mirror to $\PP^2$ with a smooth elliptic curve is given by the compactification $Y$. From the point of view of \emph{Homological Mirror Symmetry}, see for instance Auroux--Katzarkov--Orlov \cite{AKO}. From the algebro-geometric point of view on mirror symmetry, within the context of the \emph{Gross--Siebert program} see for instance \cite[Example 6.2]{GHS}, and \cite{GRS,CPS} for various other aspects related to mirror symmetry for $\PP^2$.

Since $(Y,D)$ is a 2-dimensional log Calabi--Yau surface, one can consider the log Gromov--Witten counts of $\mathbb{A}^1$-curves $GW_{v,\beta}^{(Y,D)}$ as defined in \eqref{eq_log_gw}, for any contact order $v$ and curve class $\beta \in H_2(Y,\ZZ)$. 
Because $Y$ is obtained from the toric surface $\overline{Y}$ by a sequence of non-toric blow-ups, the tropicalization $\mathrm{Trop}(Y,D)$ is canonically identified with the tropicalization $\mathrm{Trop}(\overline{Y},\overline{D})$, which in turn may be identified with $\RR^2$ containing the fan of $(\overline{Y},\overline{D})$. In particular, the set $\mathrm{Trop}_\ZZ(Y,D)$ is naturally identified with $\ZZ^2$, and we therefore regard contact orders $v$ simply as elements of $\ZZ^2$.

\section{The log Gromov--Witten/Vafa--Witten correspondence}

We recall the definition of Vafa--Witten invariants of $\PP^2$ and state the correspondence theorem relating them to counts of $\mathbb{A}^1$-curves in $(Y,D)$ discussed in the previous section.
\subsection{Vafa--Witten invariants of $\PP^2$}
\label{section_vw_inv}

The Vafa--Witten invariants 
\[ VW_{r,c_1,c_2}^S \in \QQ \] of a polarized smooth projective surface $S$ are rational numbers introduced by Tanaka-Thomas \cite{TT, TT1}. They are defined via $\CC^\star$-localization on the moduli space of Gieseker semistable Higgs bundles on $S$ with Chern classes $(r,c_1,c_2)$. When $S$ is a del Pezzo surface, for instance when $S=\PP^2$, Tanaka-Thomas prove that  $VW_{r,c_1,c_2}^S$ can be described alternatively in terms of weighted Euler characteristics \cite[\S6.1]{TT}. Therefore, denoting by $K_S$ the non-compact Calabi--Yau 3-fold $K_S$ given by the total space of the canonical line bundle on $S$, we have 
\[VW_{r,c_1,c_2}^S = \overline{\Omega}_{r,c_1,c_2}^{K_S}\,,\]
where $\overline{\Omega}_{r,c_1,c_2}^{K_S}$ are the Donaldson--Thomas invariants of $K_S$ for coherent sheaves supported on $S$. On a general polarized Calabi--Yau 3-fold $X$, Donaldson--Thomas invariants $\overline{\Omega}_\gamma^X \in \QQ$ are defined in terms of weighted Euler characteristics of moduli spaces of semistable coherent sheaves on $X$ with Chern classes $\gamma$. 
These invariants were first introduced by Thomas \cite{Thomas} in the absence of strictly semistable sheaves, and later defined in the general case by Joyce–Song \cite{JoyceSong}. Joyce–Song further conjectured that, for sufficiently generic polarizations, the BPS invariants $\Omega_\gamma^X \in \QQ$ defined via the following multiple cover formula
\[ \overline{\Omega}_\gamma^X = \sum_{\substack{k \in \ZZ\\ \gamma=k\gamma'}} \frac{1}{k^2} \Omega_{\gamma'}^X\,,\]
are in fact integers \cite[Conjecture 6.12]{JoyceSong}. 

When $X=K_S$ and $S$ is a del Pezzo surface, this integrality conjecture was proved by Meinhardt  \cite{meinhardt2015donaldson}.
In this case, the moduli stack of semistable sheaves is smooth, and the BPS invariants $\Omega_\gamma^{K_S}$ admit a description in terms of the intersection cohomology of the moduli spaces of Gieseker semistable sheaves on $S$. Explicitly, denoting by $\mathcal{M}_{\gamma}^{s}$ (resp.\, $\mathcal{M}_{\gamma}^{ss}$) the coarse moduli spaces of Gieseker stable (resp\,. semistable) coherent sheaves on $S$ with Chern classes $\gamma$ \cite{Gieseker, Maruyama1, MaruyamaII, HuybrechtsLehn}, the BPS invariants satisfy
\begin{equation}
\label{eq_bps}
\Omega_{\gamma}^{K_S} = \begin{cases} 
      (-1)^{\dim \mathcal{M}_\gamma^s}\sum_{i \geq 0} (-1)^i \dim \mathrm{IH}^i(\mathcal{M}_{\gamma}^{ss},\QQ) & \mathrm{if} ~~ \mathcal{M}_{\gamma}^{s} \neq 0 \\
      0  & \mathrm{else}
   \end{cases}   \end{equation}
where $\mathrm{IH}^i$ denotes the intersection cohomology groups -- see \cite{GoreskyMacpherson,KirwanWoolf}.

Consequently, in the special case of a del Pezzo surface $S$, the Vafa--Witten invariants $VW_{r,c_1,c_2}^S$ may also be viewed as particular Donaldson--Thomas invariants $\overline{\Omega}_{r,c_1,c_2}^{K_S}$, and can be concretely described via \eqref{eq_bps} in terms of intersection cohomology of the moduli spaces of Gieseker semistable sheaves.
In this paper, we focus on the case $S=\PP^2$. While the definition of Gieseker stability generally depends on a choice of polarization, this dependence disappears when $S=\PP^2$ since any ample line bundle is a power of $\mathcal{O}_{\PP^2}(1)$. 

Moduli spaces of Gieseker semistable sheaves on $\PP^2$ have been studied extensively. 
The moduli space of stable sheaves $\mathcal{M}_{r,c_1,c_2}^{s}$ is a smooth, irreducible, quasi-projective variety of dimension 
\[ \dim \mathcal{M}_{r,c_1,c_2}^s = r^2+c_1^2+3rc_1-2\chi r +1\,,\] 
where $\chi:=r+\frac{1}{2}c_1(c_1+3)-c_2$ is the holomorphic Euler characteristic of a coherent sheaf with Chern classes $(r,c_1,c_2)$. When $\mathcal{M}_{r,c_1,c_2}^{s}=\mathcal{M}_{r,c_1,c_2}^{ss}$ -- for example, when $(r,c_1,c_2)$ is a primitive vector in $\ZZ^3$ --  $\mathcal{M}_{r,c_1,c_2}^{s}$ is also projective but not in general. For arbitrary $(r,c_1,c_2)$, the moduli space of semistable sheaves $\mathcal{M}_{\gamma}^{ss}$ is a projective irreducible variety which may be singular when  $\mathcal{M}_{r,c_1,c_2}^{s}\neq \mathcal{M}_{r,c_1,c_2}^{ss}$.
The irreducibility of the moduli spaces $\mathcal{M}_{r,c_1,c_2}^{ss}$ is a special feature of $\PP^2$ and does not hold for general surfaces;  see \cite{DrezetPotier} for a proof in this case.

\subsection{The log Gromov--Witten/Vafa--Witten correspondence}
\label{sec:log GW/VW}
The correspondence theorem between Vafa--Witten invariants of $\PP^2$ and log Gromov--Witten invariants of $(Y,D)$ is a particular consequence of the general correspondence theorem between quiver DT invariants and log Gromov--Witten invariants of cluster varieties established in \cite{ABquiverscurves}. We review this, focusing on $\PP^2$ below. 

We first explain how the numerical data given by $\gamma = (r,c_1,c_2) \in \ZZ^3$ in the previous section determines a  contact order $v_{\gamma} \in \ZZ^2$ and a curve class $\beta_{\gamma} \in H_2(Y,\ZZ)$. We let
\begin{align}
\label{eq:n123}
n_1 & = - \chi = -r - \frac{1}{2} c_1(c_1+3) +c_2 \,, \\
\nonumber
n_2 & = r + c_1 - \chi = c_1 - \frac{1}{2}c_1(c_1+3)+c_2 \,, \\
\nonumber
n_3 & = r + 2c_1 - \chi =2c_1 -\frac{1}{2}c_1(c_1+3)+c_2  \,.    
\end{align}
Note that $(n_1,n_2,n_3)$ is the dimension vector of a quiver representation corresponding to a coherent sheaf with Chern classes  $(r,c_1,c_2)$ -- see \cite[\S4]{ABquiverscurves} for details. 
Recall that the ray generators of the fan $\Sigma$ defining the toric variety $\overline{Y}$ include $(1,1),(-2,1),(1,-2)$, which correspond to the divisors along which we do non-toric blow-ups to obtain $Y$. We will denote by 
\[  \pi \colon Y \longrightarrow \overline{Y} \,, \]
the blow-up map. Tropical balancing for $\mathbb{A}^1$-curves leads to the definition of the contact order $v_{\gamma} \in \ZZ^2$ by
\begin{equation}
\label{eq: vgamma}
    n_1(1,1) + n_2(-2,1)  + n_3(1,-2) + v_{\gamma} = 0 \,. 
\end{equation}
Hence, we have
\begin{equation}
\label{eq:vgamma2}
v_{\gamma} = (r,r+3c_1) \,,
\end{equation} 
which depends only on $r$ and $c_1$ and not on $c_2$, which will be crucial when discussing mock modularity of generating series of log Gromov--Witten invariants in the next section. Since $\Sigma$ is a complete fan, the vector $v_{\gamma} \in \ZZ^2$ can be uniquely written in the form $v_{\gamma} = a w_1 + b w_2$, for $a,b \in \ZZ_{\geq 0}$ and $w_1,w_2 \in \mathcal{R}$ which span a $2$-dimensional cone of $\Sigma$. As explained in joint work with Gross \cite[Equations (2.18)-(2.19)]{HDTV}, associated to the data of a set of vectors summing up to zero such as in \eqref{eq: vgamma}, there exists a unique toric curve class $\overline{\beta}_{\gamma} \in H_2(\overline{Y},\ZZ)$ satisfying:
\[  \overline{\beta}_{\gamma} \cdot D_w = n_1 \delta_{w,(1,1)} + n_2  \delta_{w,(-2,1)} + n_3 \delta_{w,(1,-2)} + a \delta_{w,w_1} + b \delta_{w,w_2}   \]
where $D_w$ stands for the divisor corresponding to a ray $\RR_{\geq 0} w$ with $w \in \mathcal{R}$ given in \eqref{eq: rays}. Then, we define
\begin{equation}
\label{eq:betagamma}
     \beta_{\gamma} = \pi^* \overline{\beta}_{\gamma} - n_1 E_1 - n_2E_2 -n_3 E_3  \in H_2(Y,\ZZ) \,,
\end{equation}
where $E_1,E_2$ and $E_3$ denote the exceptional divisors obtained by blowing-up the points $x_{(1,1)}, x_{(1,-1)}$ and $x_{(-2,1)}$ respectively, as in \S\ref{sec:rational elliptic}.

In joint work with Bousseau, we proved the following \cite[Theorem 4.1]{ABquiverscurves} \footnote{The statement of \cite[Theorem 4.1]{ABquiverscurves} is slightly wrong, and contains an additional numerical factor not included in \eqref{eq:right}. The problem is due to the fact that when comparing the three dimensional scattering diagram associated to the quiver for $K_{\PP^2}$ with the two dimensional scattering diagram associated to $(Y,D)$, there comes in an extra factor, similar to \cite[Definition 4.5 and Definition 4.7]{GHS}. The rest of the proof remains the same.} relating the Vafa--Witten invariants $VW_{r,c_1,c_2}^{\PP^2}$ reviewed in \S \ref{section_vw_inv} and the log Gromov--Witten counts $GW_{v,\beta}^{(Y,D)}$ of $(Y,D)$ defined in \S\ref{sec:rational elliptic}:

\begin{theorem}
\label{Thm:AB}
For any $\gamma = (r,c_1,c_2) \in \ZZ^3$ such that $r > 0$ and $-1 < \frac{c_1}{r} \leq 0 $, we have
\begin{equation}
\label{eq:right}
     VW_{r,c_1,c_2}^{\PP^2} = (-1)^{\dim \mathcal{M}_{\gamma}^s} \cdot  GW_{v_{\gamma},\beta_{\gamma}}^{(Y,D)} \,,
\end{equation}
    where the contact order $v_{\gamma}\in \ZZ^2$, and the curve class $\beta_{\gamma} \in H_2(Y,\ZZ)$ are given as in \eqref{eq:vgamma2} and \eqref{eq:betagamma}, respectively.
\end{theorem}

The condition $-1 < \frac{c_1}{r} \leq 0 $ in the statement of Theorem~ \ref{Thm:AB} may be assumed without loss of generality, since tensoring coherent sheaves contributing to $ \overline{\Omega}_{\gamma}$ by a line bundle $\mathcal{O}(k)$, for some $k \in \ZZ$
, leads to a moduli space isomorphic to $ \mathcal{M}_{\gamma}^s$, and tensoring by $\mathcal{O}(k)$ changes $\frac{c_1}{r}$ to $\frac{c_1}{r} + k$.

For a heuristic explanation for why the correspondence in Theorem \ref{Thm:AB} is natural see \cite[\S4.1]{ABquiverscurves} or \cite[\S8]{Bousseau-Takahashi}. Roughly speaking, Homological Mirror Symmetry suggests that stable sheaves on $\PP^2$ should correspond to special Lagrangian submanifolds in $Y \setminus D$ with boundaries on an elliptic curve, arising as a fiber of $\phi_{|-K_Y|}$. See Thomas--Yau \cite{TY} for the general discussion on the relation between stable sheaves and special Lagrangians under mirror symmetry. After an appropriate hyper-K\"ahler rotation of the complex structure on $Y \setminus D$, these special Lagrangians should turn into holomorphic curves, now with boundary on a fiber of an SYZ special Lagrangian fibration on $Y \setminus D$. The log Gromov--Witten invariants of $(Y,D)$ are algebro-geometric analogues of such holomorphic curves.

\section{Mock modularity of log Gromov--Witten generating series}
After discussing mock modularity in \S\ref{sec: mock modularity} and known mock modularity results for generating series of Vafa--Witten invariants of $\PP^2$ in \S\ref{sec_mock_vw}, we describe generating series of log Gromov--Witten invariants and prove their mock modularity in \S\ref{subsec: mocklog}.

\subsection{Modularity and mock modularity}
\label{sec: mock modularity}
The following is the standard definition of a modular form with integer weights.

\begin{definition}
\label{defn:modular fnc}
Let $\mathbb{H} = \{  \tau \in \CC ~ | ~ \mathrm{Im} \tau > 0 \}$ be the upper half-plane. A \emph{modular form of weight $k \in \ZZ$} is a holomorphic function $f \colon \mathbb{H} \to  \mathbb{C}$ satisfying the following property:
\begin{equation}
\label{eq: modular}
    f\left(\frac{a\tau+b}{c \tau+d}\right)=(c\tau+d)^k f(\tau) ~~
\mathrm{for ~ every} ~ 
\begin{pmatrix}
 a & b \\
 c & d
\end{pmatrix}
\in \mathrm{SL}_2(\ZZ) \,.
\end{equation}
\end{definition}

One can analogously define modular forms with \emph{fractional weights}, by allowing an extra phase in \eqref{eq: modular} -- see for instance \cite[\S3.2]{Manschot}.

\begin{example}
The \emph{Dedekind eta function} defined by
\begin{equation} \label{eq_eta}
\eta(\tau):= q^{\frac{1}{24} } \prod_{n=1}^\infty (1-q^n)=q^{\frac{1}{24}}
(1-q-q^2+q^5+q^7-q^{12}+\dots)
\end{equation}
is a modular form of weight $1/2$ for $\mathrm{SL}_2(\ZZ)$.
\end{example}

One can also define \emph{vector-valued modular forms} as follows.
\begin{definition}
\label{def:vector valued}
A \emph{$d$-dimensional vector-valued modular form of weight $k$} is a vector of holomorphic functions 
\[ \vec{f} = \begin{pmatrix}
f_0 \\
\vdots \\
f_{d-1}
\end{pmatrix} \colon \mathbb{H} \to \mathbb{C}^d \,,
  \]
satisfying the following property:
\begin{align*}
\vec{f} \begin{pmatrix} \frac{-1}{\tau} \end{pmatrix} & = \tau^k M_1 \vec{f}(\tau), ~ \mathrm{and} \\
    \vec{f} (\tau + 1 ) & = M_2 \vec{f}(\tau) \,,
\end{align*}
for some $d \times d$ matrices $M_1,M_2$.
\end{definition}

Mock modular forms are holomorphic functions that fail to be modular, but whose modularity is recovered upon adding an appropriate non-holomorphic completion. The following definition can be found in \cite[\S 3.3]{Manschot}, where the authors consider more general mock modular forms with an additional parameter $\ell$, which we set to be equal to zero -- see also \cite{zwegers2008mock}.

\begin{definition}
\label{defn: mock modular}
Let $k \in \frac{1}{2} \ZZ$. A \emph{mock modular form} of weight $k$ is a holomorphic function $f \colon \mathbb{H} \to \mathbb{C}$ for which there exists a modular form $g$ of weight $2 -k$ such that the non-holomorphic function 
\[ \hat{f} (\tau,\overline{\tau}) \coloneqq f(\tau) -2^{1-k}i\int_{- \overline{\tau}}^{i\infty} \frac{g(v)}{(-i(v + \tau))^k} dv  \]
transforms as a modular form of weight $k$, that is,
\begin{align*}
  \hat{f}\left(\frac{a\tau+b}{c \tau+d} , \frac{a \overline{\tau}+b}{c \overline{\tau}+d}  \right ) & = \epsilon(M) (c \tau + d )^k \hat{f} (\tau, \overline{\tau}) \,,  
\end{align*}
for every $
M = \begin{pmatrix}
 a & b \\
 c & d
\end{pmatrix}
\in \mathrm{SL}_2(\ZZ)$, and for some $\epsilon(M) \in \CC^*$.
The function $\hat{f}$ is known as the \emph{non-holomorphic completion} of $f$, and the associated function $g$ is referred to as the \emph{shadow} of $f$.
\end{definition}

A mock modular form in the sense of Definition~\ref{defn: mock modular} is also referred to as a mock modular form of depth~$1$. In the following sections, we will study generating series of enumerative invariants whose mock modularity properties are governed by mock modular forms of higher depth, defined recursively as follows.

\begin{definition}
\label{defn:mock modular higher weight}
    Let $k \in \frac{1}{2} \ZZ$. A \emph{mock modular form} of weight $k$ and depth $r$ is a holomorphic function $f \colon \mathbb{H} \to \mathbb{C}$ such that there exists functions $\hat{f}_1(\tau,\overline{\tau}),\ldots, \hat{f}_n(\tau,\overline{\tau}) $, which are non-holomorphic modular completions of mock modular forms of weight $k$ and depth $r-1$, and $g_1(\tau), \ldots,g_n(\tau)$, which are modular of weight $2-k$, such that the non-holomorphic function 
\[ \hat{f} (\tau,\overline{\tau}) \coloneqq f(\tau) -2^{1-k}i \sum_{j=1} ^n\int_{- \overline{\tau}}^{i\infty} \frac{\hat{f}_j(\tau-v)g_j(v)}{(-i(v + \tau))^k} dv  \]
transforms as a modular form of weight $k$, that is,
\begin{align*}
  \hat{f}\left(\frac{a\tau+b}{c \tau+d} , \frac{a \overline{\tau}+b}{c \overline{\tau}+d}  \right ) & = \epsilon(M) (c \tau + d )^k \hat{f} (\tau, \overline{\tau}) \,,  ~~ 
\end{align*}
for every $
M = \begin{pmatrix}
 a & b \\
 c & d
\end{pmatrix}
\in \mathrm{SL}_2(\ZZ)$, and for some $\epsilon(M) \in \CC^*$.
\end{definition}

One can also define vector-valued mock modular forms, analogously as in Definition \ref{def:vector valued}.

\subsection{Mock modularity and Vafa--Witten invariants of $\PP^2$.}
\label{sec_mock_vw}

The conjectural S-duality of $\mathcal{N}=4$ super Yang-–Mills theory in theoretical physics predicts that generating series of Vafa-–Witten invariants of smooth projective surfaces exhibit modular behavior \cite{vafawitten}. When the surface 
$S$ satisfies $h^{2,0}(S)>0$, these generating series are expected to be genuine modular forms; see, for example, \cite{GKrefined, GKL, JK} for recent progress in this direction.

In contrast, for surfaces with $h^{2,0}(S)=0$, the modular transformation properties are expected to be more intricate and typically involve mock modular forms. In the case of del Pezzo surfaces, it is predicted that the generating series of rank 
$r$ Vafa–Witten invariants assemble into vector-valued mock modular forms of weight 
$-\frac{\chi}{2}$ and depth $r-1$, where $\chi$ denotes the topological Euler characteristic of the surface \cite{alexandrov2025mock, AP, Manschot}.
We briefly review some existing results on the mock modularity of generating series of Vafa--Witten invariants of $\PP^2$ \cite{Manschot}.

For every $r \in \ZZ_{>0}$ and $c_1\in \ZZ$, we consider the \emph{Vafa--Witten generating series}
\begin{equation}
\label{eq: vw generating}
   h^{VW}_{r,c_1}(\tau):= \sum_{c_2 \in \ZZ} (-1)^{\dim \mathcal{M}^s_{r,c_1,c_2}} VW_{r,c_1,c_2}^{\PP^2} q^{c_2+\frac{c_1^2}{2r}-\frac{r}{8}}\,,  
\end{equation}
where $q=e^{2i\pi \tau}$, and $VW_{r,c_1,c_2}^{\PP^2}$ are the Vafa--Witten invariants of $\PP^2$ -- see \cite[Equation (6.2)]{Manschot}. 
Since tensoring a rank $r$ coherent sheaf with the line bundle
 $\mathcal{O}(k)$ shifts its first Chern class from 
 $c_1$ to $c_1+kr$, it follows that
$h^{VW}_{r,c_1}=h^{VW}_{r,c_1+kr}$ for all $k \in \ZZ$, and so it is sufficient to consider the series $h^{VW}_{r,c_1}(\tau)$ for $0 \leq c_1 <r$. 

The vector of functions $(h^{VW}_{r,c_1}(\tau))_{0 \leq c_1<r}$ is expected to transform under the action of $SL(2,\ZZ)$ on the variable $\tau$ as a vector-valued mock modular form of weight $-\frac{3}{2}$ and depth $r-1$  \cite{AP, alexandrov2025mock, Manschot}. This has been checked by explicit calculations for $r \leq 3$ \cite{Manschot}. For $r=1$, $h_{1,0}(\tau)$ is the generating series of Euler characteristics of Hilbert schemes of points on $\PP^2$, and so is expressed by G\"ottsche formula \cite{gottsche} in terms of the eta function $\eta(\tau)$ reviewed in
\eqref{eq_eta} as
\[h_{1,0}(\tau) = \frac{1}{\eta(\tau)^3}\,,\]
which is indeed a modular form of weight $-\frac{3}{2}$.
For $r=2$ and $c_1\in \{0,1\}$, we have \cite{yoshioka, klyachko},
\[ h_{2,c_1}(\tau) = \frac{f_{2,c_1}(\tau)}{\eta(\tau)^6}\,,\]
where 
\[ f_{2,c_1}(\tau)=3 \sum_{n \geq 0} H(4n-c_1)q^{n-\frac{c_1}{4}} \,,\]
where $H(m)$ is the Hurwitz class number, that is the weighted count of equivalence classes of integral binary quadratic forms of discriminant 
$m$. Generating functions of Hurwitz numbers were historically among the first examples of mock modular forms \cite{zagier, zwegers2008mock}. 
Indeed, the vector $(f_{2,c_1}(\tau))_{0 \leq c_1 \leq 1}$ is a vector-valued mock modular form of weight $\frac{3}{2}$ and depth $1$, with non-holomorphic modular completion given by
\[ \hat{f}_{2,c_1}(\tau, \overline{\tau})=f_{2,c_1}(\tau)-\frac{3i}{4\sqrt{2}\pi} \int_{-\overline{\tau}}^{i\infty} \frac{\vartheta_{\frac{c_1}{2}}(v)}{(-i(v+\tau))^{\frac{3}{2}}}dv \,,\]
where the shadow $\vartheta_{\frac{c_1}{2}}(\tau)$
is the theta function
\[ \vartheta_{\frac{c_1}{2}}(\tau) = \sum_{n \in \Z +\frac{c_1}{2}} q^{\frac{n^2}{2}}\,,\]
which is modular of weight $2-\frac{3}{2}=\frac{1}{2}$.
Since $(f_{2,c_1}(\tau))_{0 \leq c_1 \leq 1}$ is a vector-valued mock modular form of weight $\frac{3}{2}$ and $\eta(\tau)$ is modular of weight $\frac{1}{2}$, $(h_{2,c_1}(\tau))_{0 \leq c_1 \leq 1}$ is a vector-valued mock modular form of weight $\frac{3}{2}-6 \times \frac{1}{2}=-\frac{3}{2}$.

For $r=3$, the series $h_{3,c_1}(\tau)$ have been computed by Manschot \cite{manschot1, manschot2}, who also established their mock-modular behavior \cite{Manschot}. The result is of the form 
\[ h_{3,c_1}(\tau)=\frac{f_{3,c_1}(\tau)}{\eta(\tau)^9}\,,\]
where the function $f_{3,c_1}(\tau)$ is given in a complicated but explicit way in \cite[\S 6.3]{Manschot}. It is proved in \cite{Manschot} that $(f_{3,c_1}(\tau))_{0 \leq c_1 \leq 2}$ is a vector-valued mock modular form of weight $\frac{3}{2}$ and depth $2$. For instance, the non-holomorphic modular completion of $f_{3,0}(\tau)$ is given by 
\[ \hat{f}_{3,0}(\tau, \overline{\tau})= f_{3,0}(\tau)-\frac{i}{\pi}\left(\frac{3}{2}\right)^{\frac{3}{2}} \sum_{j=0}^1 
\int_{-\overline{\tau}}^{i\infty} 
\frac{\hat{f}_{2,j}(\tau,-v) \vartheta_{\frac{j}{2}}(3v)}{(-i(\tau+v))^{\frac{3}{2}}}dv\,.\]
For general $r$, the series $h^{VW}_{r,c_1}(\tau)$ can be calculated and expressed in terms of indefinite theta functions of a lattice of indefinite signature $(r,r)$ \cite{manschot2}. However, a detailed proof of their mock modularity properties does not seem to have been completed yet.

\subsection{Mock modularity of log Gromov--Witten invariants of the mirror $(Y,D)$ to $\PP^2$}
\label{subsec: mocklog}

We define a generating series of log Gromov--Witten invariants as follows:

\begin{definition}
Let $(Y,D)$ be the rational elliptic surface mirror to $\PP^2$ as in \S\ref{sec:rational elliptic}. Fix $r \in \ZZ_{\geq 0}$ and $c_1 \in \ZZ$. Let $v= (r,r+3c_1) \in \ZZ^2$ and $\beta_{(r,c_1,0)} \in H_2(Y,\ZZ)$ be the contact order and curve class obtained from the numerical data $\gamma = (r,c_1,c_2=0) \in \ZZ^3$, as in \S\ref{sec:log GW/VW}. Define the \emph{log Gromov--Witten generating series} for $(Y,D)$ by
\begin{equation}
\label{Eq:log GW generating}
     h^{GW}_{r,c_1}(\tau):= \sum_{c_2 \in \ZZ} GW^{(Y,D)}_{(r,r+3c_1),\beta_{(r,c_1,0)} + c_2F } q^{c_2 + \frac{c_1^2}{2r} - \frac{r}{8} } \,,
\end{equation}
where $F$ denotes the class of an elliptic fiber of the map  $\phi_{|-K_Y|}  \colon Y \to \PP^1$ as in \eqref{eq:fibration}, and $q= e^{2i\pi\tau}$.
\end{definition}

\begin{theorem}
\label{thm:logGW/VW generating}
    For any $r \in \ZZ_{\geq 0}$ and $c_1 \in \ZZ$ with $-1 < \frac{c_1}{r} \leq 0$, the log Gromov--Witten generating series for $(Y,D)$ and the Vafa--Witten generating series for $\PP^2$, given in Equations \eqref{Eq:log GW generating} and \eqref{eq: vw generating} respectively are equal:
   \[ h^{GW}_{r,c_1}(\tau) = h^{VW}_{r,c_1}(\tau) \,. \]
\end{theorem}

\begin{proof}
By Theorem \ref{Thm:AB}, we have for any $\gamma = (r,c_1,c_2)$ with $-1 < \frac{c_1}{r} \leq 0$:
\[ (-1)^{\dim \mathcal{M}_{\gamma}^s} VW_{\gamma}^{\PP^2} = GW^{(Y,D)}_{v_{\gamma},\beta_{\gamma}} \,. \]
As explained in \S\ref{sec:log GW/VW}, we have $v_{\gamma} = (r,r+3c_1) $, and it follows from the definition of $\beta_{\gamma}$ that
    \[  \beta_{\gamma} = \beta_{(r,c_1,c_2)} = \beta_{(r,c_1,0)} + c_2 \beta_{(0,0,1)}  \,. \]
Moreover, by \eqref{eq:n123} we have $\beta_{(0,0,1)} = F$. Hence, 
\[ GW^{(Y,D)}_{v_{\gamma},\beta_{\gamma}} = GW^{(Y,D)}_{(r,r+3c_1),\beta_{(r,c_1,0)} + c_2F } \,, \] 
and hence the equality of the log Gromov--Witten and Vafa--Witten generating series, given in Equations \eqref{Eq:log GW generating} and \eqref{eq: vw generating}, follows. 
\end{proof}

The following is an immediate corollary of Theorem \ref{thm:logGW/VW generating}, together with the known mock modularity results of generating series of Vafa--Witten invariants of $\PP^2$ reviewed in \S\ref{sec_mock_vw}.

\begin{corollary}
\label{cor:proof of mock log}
Let $(Y,D)$ be the rational elliptic surface mirror to $\PP^2$ as in \S\ref{sec:rational elliptic}.
For any $1 \leq r \leq 3$, the vector $(h^{GW}_{r,c_1}(\tau))_{-r< c_1 \leq 0}$, with components the log Gromov--Witten generating series of $(Y,D)$ given in \eqref{Eq:log GW generating}, is a vector-valued mock modular form of weight $-\frac{3}{2}$ and depth $r$.
\end{corollary}

More generally, we expect mock modularity of generating series of Vafa–Witten invariants to translate into the following conjecture for logarithmic Gromov–Witten invariants on the mirror to $\PP^2$:

\begin{conjecture}
\label{conj: mirror P2}
Let $(Y,D)$ be the rational elliptic surface mirror to $\PP^2$ as in \S\ref{sec:rational elliptic}.
For any $r \in \ZZ_{> 0}$, the vector $(h^{GW}_{r,c_1}(\tau))_{-r< c_1 \leq 0}$, with components the log Gromov--Witten generating series of $(Y,D)$ given in \eqref{Eq:log GW generating}, is a vector-valued mock modular form of weight $-\frac{3}{2}$ and depth $r$.
\end{conjecture}

\bibliographystyle{plain}
\bibliography{bibliography}

\begin{thebibliography}{10}

\bibitem{logGWbyAC}
Dan Abramovich and Qile Chen.
\newblock Stable logarithmic maps to {D}eligne-{F}altings pairs {II}.
\newblock {\em Asian Journal of Mathematics}, 18(3):465--488, 2014.

\bibitem{ACGS}
Dan Abramovich, Qile Chen, Mark Gross, and Bernd Siebert.
\newblock {\em Punctured logarithmic maps}, volume~15 of {\em Memoirs of the
  European Mathematical Society}.
\newblock EMS Press, Berlin, 2025.

\bibitem{alexandrov2025mock}
Sergei Alexandrov.
\newblock Mock modularity at work, or black holes in a forest.
\newblock {\em arXiv preprint arXiv:2505.02572}, 2025.

\bibitem{AP}
Sergei Alexandrov and Boris Pioline.
\newblock Black holes and higher depth mock modular forms.
\newblock {\em Comm. Math. Phys.}, 374(2):549--625, 2020.

\bibitem{ABquiverscurves}
H\"ulya Arg\"uz and Pierrick Bousseau.
\newblock Quivers and curves in higher dimension.
\newblock {\em Trans. Amer. Math. Soc.}, 378(1):389--420, 2025.

\bibitem{HDTV}
H\"ulya Arg\"uz and Mark Gross.
\newblock The higher-dimensional tropical vertex.
\newblock {\em Geom. Topol.}, 26(5):2135--2235, 2022.

\bibitem{AKO}
Denis Auroux, Ludmil Katzarkov, and Dmitri Orlov.
\newblock Mirror symmetry for del {P}ezzo surfaces: vanishing cycles and
  coherent sheaves.
\newblock {\em Invent. Math.}, 166(3):537--582, 2006.

\bibitem{Bousseau-Takahashi}
Pierrick Bousseau.
\newblock A proof of {N}. {T}akahashi's conjecture for {$(\Bbb P^2,E)$} and a
  refined sheaves/{G}romov-{W}itten correspondence.
\newblock {\em Duke Math. J.}, 172(15):2895--2955, 2023.

\bibitem{bring2}
Kathrin Bringmann, Jonas Kaszi\'an, and Larry Rolen.
\newblock Indefinite theta functions arising in {G}romov-{W}itten theory of
  elliptic orbifolds.
\newblock {\em Camb. J. Math.}, 6(1):25--57, 2018.

\bibitem{bring1}
Kathrin Bringmann, Larry Rolen, and Sander Zwegers.
\newblock On the modularity of certain functions from the {G}romov-{W}itten
  theory of elliptic orbifolds.
\newblock {\em R. Soc. Open Sci.}, 2:150310, 12, 2015.

\bibitem{CPS}
Michael Carl, Max Pumperla, and Bernd Siebert.
\newblock A tropical view on {L}andau-{G}inzburg models.
\newblock {\em Acta Math. Sin. (Engl. Ser.)}, 40(1):329--382, 2024.

\bibitem{Chen}
Qile Chen.
\newblock Stable logarithmic maps to {D}eligne-{F}altings pairs {I}.
\newblock {\em Annals of Mathematics}, pages 455--521, 2014.

\bibitem{DPW}
Atish Dabholkar, Pavel Putrov, and Edward Witten.
\newblock Duality and mock modularity.
\newblock {\em SciPost Phys.}, 9(5):Paper No. 072, 45, 2020.

\bibitem{dijkgraaf1995mirror}
Robert Dijkgraaf.
\newblock Mirror symmetry and elliptic curves.
\newblock In {\em The Moduli Space of Curves}, volume 129 of {\em Progress in
  Mathematics}, pages 149--163. Birkhäuser / Springer, 1995.

\bibitem{DrezetPotier}
Jean-Marc Drezet and Joseph Le~Potier.
\newblock Fibr\'es stables et fibr\'es exceptionnels sur {${\bf P}_2$}.
\newblock {\em Ann. Sci. \'Ecole Norm. Sup. (4)}, 18(2):193--243, 1985.

\bibitem{Gieseker}
David Gieseker.
\newblock On the moduli of vector bundles on an algebraic surface.
\newblock {\em Ann. of Math. (2)}, 106(1):45--60, 1977.

\bibitem{GoreskyMacpherson}
Mark Goresky and Robert MacPherson.
\newblock Intersection homology theory.
\newblock {\em Topology}, 19(2):135--162, 1980.

\bibitem{gottsche}
Lothar G\"ottsche.
\newblock The {B}etti numbers of the {H}ilbert scheme of points on a smooth
  projective surface.
\newblock {\em Math. Ann.}, 286(1-3):193--207, 1990.

\bibitem{GKrefined}
Lothar G\"ottsche and Martijn Kool.
\newblock Refined {$\rm SU(3)$} {V}afa-{W}itten invariants and modularity.
\newblock {\em Pure Appl. Math. Q.}, 14(3-4):467--513, 2018.

\bibitem{GKL}
Lothar G\"ottsche, Martijn Kool, and Ties Laarakker.
\newblock {${\rm SU}(r)$} {V}afa-{W}itten invariants, {R}amanujan's continued
  fractions, and cosmic strings.
\newblock {\em Michigan Math. J.}, 75(1):3--63, 2025.

\bibitem{GGLP}
Antonella Grassi, Giulia Gugiatti, Wendelin Lutz, and Andrea Petracci.
\newblock Reflexive polygons and rational elliptic surfaces.
\newblock {\em Rend. Circ. Mat. Palermo (2)}, 72(6):3185--3221, 2023.

\bibitem{GHK}
Mark Gross, Paul Hacking, and Sean Keel.
\newblock Mirror symmetry for log {C}alabi-{Y}au surfaces {I}.
\newblock {\em Publ. Math. Inst. Hautes \'Etudes Sci.}, 122:65--168, 2015.

\bibitem{GHS}
Mark Gross, Paul Hacking, and Bernd Siebert.
\newblock Theta functions on varieties with effective anti-canonical class.
\newblock {\em Mem. Amer. Math. Soc.}, 278(1367):xii+103, 2022.

\bibitem{GPS}
Mark Gross, Rahul Pandharipande, and Bernd Siebert.
\newblock The tropical vertex.
\newblock {\em Duke Math. J.}, 153(2):297--362, 2010.

\bibitem{GSannals}
Mark Gross and Bernd Siebert.
\newblock From real affine geometry to complex geometry.
\newblock {\em Ann. of Math. (2)}, 174(3):1301--1428, 2011.

\bibitem{logGW}
Mark Gross and Bernd Siebert.
\newblock Logarithmic {G}romov-{W}itten invariants.
\newblock {\em Journal of the American Mathematical Society}, 26(2):451--510,
  2013.

\bibitem{GSintrinsic}
Mark Gross and Bernd Siebert.
\newblock Intrinsic mirror symmetry.
\newblock {\em J. Amer. Math. Soc.}, 39(2):313--451, 2026.

\bibitem{HuybrechtsLehn}
Daniel Huybrechts and Manfred Lehn.
\newblock {\em The geometry of moduli spaces of sheaves}.
\newblock Cambridge Mathematical Library. Cambridge University Press,
  Cambridge, second edition, 2010.

\bibitem{JK}
Yunfeng Jiang and Martijn Kool.
\newblock Twisted sheaves and {${\rm SU}(r)/\Bbb Z_r$} {V}afa-{W}itten theory.
\newblock {\em Math. Ann.}, 382(1-2):719--743, 2022.

\bibitem{JoyceSong}
Dominic Joyce and Yinan Song.
\newblock A theory of generalized {D}onaldson-{T}homas invariants.
\newblock {\em Mem. Amer. Math. Soc.}, 217(1020):iv+199, 2012.

\bibitem{KirwanWoolf}
Frances Kirwan and Jonathan Woolf.
\newblock {\em An introduction to intersection homology theory}.
\newblock Chapman \& Hall/CRC, Boca Raton, FL, second edition, 2006.

\bibitem{klyachko}
Alexander~A. Klyachko.
\newblock Moduli of vector bundles and numbers of classes.
\newblock {\em Funktsional. Anal. i Prilozhen.}, 25(1):81--83, 1991.

\bibitem{KS}
Maxim Kontsevich and Yan Soibelman.
\newblock Affine structures and non-{A}rchimedean analytic spaces.
\newblock In {\em The unity of mathematics}, volume 244 of {\em Progr. Math.},
  pages 321--385. Birkh\"auser Boston, Boston, MA, 2006.

\bibitem{Lau}
Siu-Cheong Lau and Jie Zhou.
\newblock Modularity of open {G}romov-{W}itten potentials of elliptic
  orbifolds.
\newblock {\em Commun. Number Theory Phys.}, 9(2):345--386, 2015.

\bibitem{JL}
Jun Li.
\newblock Stable morphisms to singular schemes and relative stable morphisms.
\newblock {\em J. Differential Geom.}, 57(3):509--578, 2001.

\bibitem{manschot1}
Jan Manschot.
\newblock The {B}etti numbers of the moduli space of stable sheaves of rank 3
  on {$\Bbb P^2$}.
\newblock {\em Lett. Math. Phys.}, 98(1):65--78, 2011.

\bibitem{manschot2}
Jan Manschot.
\newblock Sheaves on {$\Bbb P^2$} and generalized {A}ppell functions.
\newblock {\em Adv. Theor. Math. Phys.}, 21(3):655--681, 2017.

\bibitem{Manschot}
Jan Manschot.
\newblock Vafa-{W}itten theory and iterated integrals of modular forms.
\newblock {\em Comm. Math. Phys.}, 371(2):787--831, 2019.

\bibitem{Maruyama1}
Masaki Maruyama.
\newblock Moduli of stable sheaves. {I}.
\newblock {\em J. Math. Kyoto Univ.}, 17(1):91--126, 1977.

\bibitem{MaruyamaII}
Masaki Maruyama.
\newblock Moduli of stable sheaves. {II}.
\newblock {\em J. Math. Kyoto Univ.}, 18(3):557--614, 1978.

\bibitem{meinhardt2015donaldson}
Sven Meinhardt.
\newblock Donaldson-thomas invariants vs. intersection cohomology for
  categories of homological dimension one.
\newblock {\em arXiv preprint arXiv:1512.03343}, 2015.

\bibitem{OPelliptic}
Georg Oberdieck and Aaron Pixton.
\newblock Holomorphic anomaly equations and the {I}gusa cusp form conjecture.
\newblock {\em Invent. Math.}, 213(2):507--587, 2018.

\bibitem{OP}
Georg Oberdieck and Aaron Pixton.
\newblock Gromov-{W}itten theory of elliptic fibrations: {J}acobi forms and
  holomorphic anomaly equations.
\newblock {\em Geom. Topol.}, 23(3):1415--1489, 2019.

\bibitem{OkPa1}
Andrey Okounkov and Rahul Pandharipande.
\newblock Gromov-{W}itten theory, {H}urwitz theory, and completed cycles.
\newblock {\em Ann. of Math. (2)}, 163(2):517--560, 2006.

\bibitem{OkPa2}
Andrey Okounkov and Rahul Pandharipande.
\newblock Virasoro constraints for target curves.
\newblock {\em Invent. Math.}, 163(1):47--108, 2006.

\bibitem{SYZ}
Andrew Strominger, Shing-Tung Yau, and Eric Zaslow.
\newblock Mirror symmetry is {$T$}-duality.
\newblock {\em Nuclear Phys. B}, 479(1-2):243--259, 1996.

\bibitem{TT}
Yuuji Tanaka and Richard~P. Thomas.
\newblock Vafa-{W}itten invariants for projective surfaces {II}: semistable
  case.
\newblock {\em Pure Appl. Math. Q.}, 13(3):517--562, 2017.

\bibitem{TT1}
Yuuji Tanaka and Richard~P. Thomas.
\newblock Vafa-{W}itten invariants for projective surfaces {I}: stable case.
\newblock {\em J. Algebraic Geom.}, 29(4):603--668, 2020.

\bibitem{TY}
R.~P. Thomas and S.-T. Yau.
\newblock Special {L}agrangians, stable bundles and mean curvature flow.
\newblock {\em Comm. Anal. Geom.}, 10(5):1075--1113, 2002.

\bibitem{Thomas}
Richard~P. Thomas.
\newblock A holomorphic {C}asson invariant for {C}alabi-{Y}au 3-folds, and
  bundles on {$K3$} fibrations.
\newblock {\em J. Differential Geom.}, 54(2):367--438, 2000.

\bibitem{vafawitten}
Cumrun Vafa and Edward Witten.
\newblock A strong coupling test of {$S$}-duality.
\newblock {\em Nuclear Phys. B}, 431(1-2):3--77, 1994.

\bibitem{GRS}
Michel van Garrel, Helge Ruddat, and Bernd Siebert.
\newblock Intrinsic enumerative mirror symmetry: {T}akahashi's log mirror
  symmetry for {$(\Bbb P^2,E)$} revisited.
\newblock {\em J. Math. Study}, 58(4):575--609, 2025.

\bibitem{yoshioka}
Kota Yoshioka.
\newblock The {B}etti numbers of the moduli space of stable sheaves of rank
  {$2$} on {$\bold P^2$}.
\newblock {\em J. Reine Angew. Math.}, 453:193--220, 1994.

\bibitem{zagier}
Don Zagier.
\newblock Nombres de classes et formes modulaires de poids {$3/2$}.
\newblock {\em C. R. Acad. Sci. Paris S\'er. A-B}, 281(21):Ai, A883--A886,
  1975.

\bibitem{zwegers2008mock}
Sander Zwegers.
\newblock Mock theta functions.
\newblock {\em arXiv preprint arXiv:0807.4834}, 2008.

\end{thebibliography}

\end{document}